\documentclass[a4paper,draft,reqno,12pt]{amsart}
\usepackage[english]{babel}
\usepackage{amsmath}
\usepackage{amssymb}
\usepackage{amscd}
\usepackage{amsthm}
\usepackage{euscript}
\usepackage{mathtools}
\usepackage[all]{xy}
\usepackage[usenames]{color}
\usepackage{colortbl}
\newtheorem{proposition}{Proposition}
\newtheorem{conjecture}{Conjecture}

\newtheorem{lemma}{Lemma}
\newtheorem{theorem}{Theorem}

\theoremstyle{definition}

\newtheorem{example}{Example}
\theoremstyle{remark}

\def\diag{{\rm diag}}
\def\SAut{{\rm SAut}}

\def\Aut{{\rm Aut}}
\def\dim{{\rm dim\,}}

\def\A{\mathbb {A}}
\def\K{\mathbb {K}}
\def\G{\mathbb {G}}

\def\a{\lambda}
\def\sl{\mathfrak{sl}}
\def\g{\mathfrak{g}}
\def\h{\mathfrak{h}}
\def\sp{\mathfrak{sp}_{2n}(\mathbb {K})}

\begin{document}
	
	\date{}
	\title[On nilpotent generators of the symplectic Lie algebra]{On nilpotent generators of the symplectic Lie algebra}
	\author{Alisa Chistopolskaya}
	\address{Higher School of Economics, Faculty of Computer Science, Kochnovsky Proezd 3, Moscow, 125319 Russia}
	\email{achistopolskaya@gmail.com}
	
	
	\thanks{The research was supported by the grant RSF 19-11-00172.}

	\begin{abstract}
		Let $\sp$ be the symplectic Lie algebra over an algebraically closed field of characteristic zero. We prove that for any nonzero nilpotent element $X \in \sp$ there exists a nilpotent element $Y \in \sp$ such that $X$ and $Y$ generate $\sp$. 
	\end{abstract}
	\maketitle
	
	\section{Introduction}
	It is an important problem to find a minimal generating set of a given algebra. This problem was studied actively for semisimple Lie algebras. In 1951, Kuranishi \cite{MK} observed that any semisimple Lie algebra over a field of characteristic zero can be generated by two elements. Twenty-five years later, Ionescu \cite{TI} proved that for any nonzero element $X$ of a complex or real simple Lie algebra $\g$ there exists an element $Y$ such that the elements $X$ and $Y$ generate the Lie algebra $\g$. 
	In 2009, Bois~\cite{JB} proved that any simple Lie algebra in characteristic different from~$2$ and $3$ can be generated by two elements. He also extended Ionescu’s result to classical simple Lie
	algebras over a field of characteristic different from~$2$ and $3$. In~\cite{AC} we obtained an analogue of Ionescu's result for nilpotent generators of the special linear Lie algebra $\sl_n$. This paper is devoted to the case of the symplectic Lie algebra over an algebraically closed field $\K$ of characteristic zero.
	
	\begin{theorem} \label{theorem}
		For any nonzero nilpotent element $X \in \sp$ there exists a nilpotent element $Y \in \sp$ such that $X$ and $Y$ generate $\sp$.
	\end{theorem}
	
	The idea of the proof is to reduce the problem to the case when $X$ is the lowest weight vector in the adjoint representation \cite[Theorem~4.3.3]{CM}. In this case we take as $Y$ an element from the principal nilpotent orbit.
	
	In the last section we discuss a conjecture generalizing Theorem~\ref{theorem} to an arbitrary simple Lie algebra. 
	We also hope for possible applications of our results to transitivity properties of the symplectomorphism group of $\A^{2n}$.
	More precisely, Conjecture 2 claims that for all~$n \geqslant 2$ one can find three one-parameter subgroups $U_1$, $U_2$, $U_3$ of symplectomorphisms of~$\A^{2n}$ 
	such that the group $H = \langle U_1, U_2, U_3 \rangle $ acts infinitely transitively on~$\A^{2n}$. 
	
	\smallskip
	
	The author is grateful to her supervisor Ivan Arzhantsev for posing the problem and permanent support and to Ernest Vinberg for useful discussions and suggestions.
	
	\section{Main results}
	
	Let $\g$ be a semisimple Lie algebra over an algebraically closed field $\K$ of characteristic zero and $\Phi$ be the root system of the algebra $\g$. Fix a Cartan subalgebra $\h$ and consider the root space decomposition with respect to $\h$: 
	$$
	\g = \h \oplus \bigoplus\limits_{\alpha \in \Phi} \g_{\alpha}. 
	$$
	
	An element $T \in \h$ is called \textit{consistent} if $\alpha(T) \neq 0$ holds for any $\alpha \in \Phi$, and for all~$\alpha$,~$\beta \in \Phi$ the condition $\alpha \neq \beta$ implies $\alpha(T) \neq \beta(T)$. Denote by $v_{\alpha}$ an arbitrary nonzero element of the line $\g_{\alpha}$.
	
	\begin{lemma} \label{Vandermonde}
		A consistent element $T$ and the element $N = \sum\limits_{\alpha \in \Phi}v_{\alpha}$ generate $\g$. 
	\end{lemma}
	
	\begin{proof}
		Consider the following $m = \dim \g - \dim \h$ elements:
		$$
		A_1 = [T, N],\quad A_i = [T, A_{i-1}],\quad i = 2,\ldots,m.
		$$ 
		Note that all elements $A_i$ belong to $\bigoplus\limits_{\alpha \in \Phi} \g_{\alpha}
		$ and $A_1,\ldots, A_m$ are linearly independent. Indeed, if we consider their coordinates in the basis $\{v_{\alpha}\}_{\alpha \in \Phi}$, we come to a Vandermonde matrix for numbers $\alpha(T)$, where $\alpha\in \Phi$. Since $T$ satisfies the consistensy conditions, it follows that this Vandermonde matrix has a nonzero determinant. 
		
		Hence, the matrices $A_i$ form a basis of the space~$\bigoplus\limits_{\alpha \in \Phi} \g_{\alpha}$. This means that the subalgebra generated by $T$ and $N$ contains all $v_{\alpha}$ and coincides with $\g$.
	\end{proof}
	
	Let $V$ be a $2n$-dimensional vector space over $\K$. The symplectic Lie algebra $\sp$ is the algebra of all linear transformations of $V$ which annihilate a non-degenerate skew-symmetric bilinear form~$\Omega$. Below in this section we consider $\g = \sp$.
	
	\begin{lemma} \label{lowestnilp}
		Theorem~\ref{theorem} holds if $X$ is the lowest weight vector.
	\end{lemma}
	
	\begin{proof}
		
		Let us fix a basis $e_1,\dots,e_{2n}$ in $V$ such that $\Omega = \begin{pmatrix}
		0 & E \\
		-E & 0
		\end{pmatrix}$. Then the subspace of all diagonal matrices in $\g$ is a Cartan
		subalgebra $\h$. Consider the set of simple roots~$\Delta$ corresponding to the matrices $E_{i, i+1} - E_{n+i+1, n+i}$, $i=1,\ldots,n-1$ and $E_{n-1, 2n}$. Choose these matrices as $v_\alpha$, $\alpha \in \Delta$.
		Denote by~$\psi$ the highest root with respect to $\Delta$.  It corresponds to the matrix unit $v_{\psi} = E_{1, n+1}$, and the lowest root $-\psi$
		corresponds to the matrix $v_{-\psi} = E_{n+1, 1}$. 
		Define
		$$
		T = (-1)^{n-1}v_{-\psi} + \sum_{\a\in \Delta} v_{\a} = 
		\begin{pmatrix}
		0 & 1   & \ldots & 0  & 0  & \ldots & 0 & 0\\
		\vdots & \vdots & \ddots &   \vdots & \vdots & \ddots & \vdots & \vdots\\
		0 & 0 &  \dots & 1  & 0  & \ldots & 0 & 0\\
		0 & 0 &  \dots & 0 & 0  & \ldots & 0 & 1 \\
		(-1)^{n-1} & 0   & \ldots & 0 & 0  & \ldots & 0 & 0\\
		0 & 0 &  \dots & 0 & -1  & \ldots & 0 & 0\\
		\vdots & \vdots & \ddots  & \vdots & \vdots & \ddots & \vdots & \vdots\\
		0 & 0  &  \dots & 0 & 0  & \ldots&-1 & 0
		\end{pmatrix}.
		$$
		Our aim is to apply Lemma~\ref{Vandermonde} to elements $T$ and $X=v_{-\psi}$.
		Note that Lemma~\ref{Vandermonde} holds if we replace $N$ by $N + H$, where $H \in \h$. Since $T^{2n} = E$, the element $T$ is semisimple. Moreover, if $\lambda$ is a root of unity of order $2n$ then 
		$$ 
		v = e_1 + \lambda e_2 + \lambda^2e_3+\dots+ \lambda^{n-1}e_n +  \frac{(-1)^{n-1}}{\lambda}e_{n+1}+ \frac{(-1)^n}{\lambda^2}e_{n+2} +\dots+ \frac{1}{\lambda^n}e_{2n}
		$$ 
		is an eigenvector of $T$ with the eigenvalue~$\lambda$. It follows that all eigenvalues of the operator~$T$ are roots of unity of order $2n$ and thus $T$ is consistent.  
		
		Since $T$ is semisimple, one can include it into a Cartan subalgebra $\widetilde{\h}$. Consider the root space decomposition of $\g$ with respect to $\widetilde{\h}$: 
			$
			\g = \widetilde{\h} \oplus \bigoplus\limits_{\beta \in \Phi} \widetilde{\g}_{\beta}. 
			$
			Let us now show that $X = \widetilde{H} + \sum\limits_{\beta\in \Phi} \widetilde{v}_{\beta}$, where $\widetilde{H} \in \widetilde{\h}$ and $\widetilde{v}_{\beta}$ is an arbitrary nonzero element of the line $\widetilde{\g}_{\beta}$. Subalgebras $\h$ and~$\widetilde{\h}$ are conjugated by an element of the group $\operatorname {Sp}_{2n}(\K)$. In other words, there exists a matrix $C \in \operatorname {Sp}_{2n}$ such that $C^{-1}TC$ is a diagonal matrix. It remains to show that all entries of the matrix $C^{-1}E_{n+1, 1}C$ outside the main diagonal are non-zero. 
		
		It sufficient to check that all entries of the $(n+1)$th column of the matrix $C^{-1}$ are non-zero, which is equivalent to the fact that $e_{n+1}$ and any $2n-1$ eigenvectors of the operator~$T$ are linearly independent. Let $\xi$ be a primitive root of unity of order $2n$. Since the matrix~$C$ consists of eigenvectors of the operator~$T$, it coincides with the matrix
		$$ 
		\begin{pmatrix}
		1 & 1 & 1 &  \ldots & 1\\
		1 & \xi & \xi^2 &  \dots & \xi^{2n-1} \\
		1 & \xi^2 & \xi^4 &  \dots & (\xi^2)^{2n-1} \\
		\vdots  & \vdots & \vdots & \ddots & \vdots \\
		1 & \xi^{n-1} & (\xi^{n-1})^2 &  \dots & (\xi^{n-1})^{2n-1}\\
		1 & \xi^{-1} & (\xi^{-1})^2 &  \dots & (\xi^{-1})^{2n-1}\\
		\vdots  & \vdots & \vdots & \ddots & \vdots \\
		1 & \xi^{-n} & (\xi^{-n})^2 &  \dots & (\xi^{-n})^{2n-1}\\
		\end{pmatrix}
		=
		\begin{pmatrix}
		1 & 1 & 1 &  \ldots & 1\\
		1 & \xi & \xi^2 &  \dots & \xi^{2n-1} \\
		1 & \xi^2 & \xi^4 &  \dots & (\xi^2)^{2n-1} \\
		\vdots  & \vdots & \vdots & \ddots & \vdots \\
		1 & \xi^{n-1} & (\xi^{n-1})^2 &  \dots & (\xi^{n-1})^{2n-1}\\
		1 & \xi^{2n-1} & (\xi^{2n-1})^2 &  \dots & (\xi^{2n-1})^{2n-1}\\
		\vdots  & \vdots & \vdots & \ddots & \vdots \\
		1 & \xi^n & (\xi^n)^2 &  \dots & (\xi^{n})^{2n-1}\\
		\end{pmatrix}
		$$
		up to permutation and scalar multiplication of rows or columns. If we replace any column of this matrix by coordinates of the vector $e_{n+1}$, we get a Vandermonde matrix for $2n-1$ different numbers from the set $\{1,\xi,\dots,\xi^{2n-1}\}$. 
		
		Using Lemma~\ref{Vandermonde} we obtain that $T$ and $X = v_{-\psi}$ generate $\sp$. So elements $X$ and $Y = T + (-1)^{n}\cdot X=\sum\limits_{\a\in \Delta} v_{\a}$ generate $\sp$ and the element $Y$ is nilpotent.
	\end{proof} 
	
	\begin{proof}[Proof of the Theorem~\ref{theorem}]
		Let $X$ be a non-zero nilpotent element of $\sp$. The closure of the orbit of $X$ under the adjoint action of $\operatorname {Sp}_{2n}(\K)$ contains the lowest weight vector $X_0$ \cite[Theorem~4.3.3]{CM}. Lemma~\ref{lowestnilp} implies that there exists a nilpotent $Y_0 \in \sp$ such that $X_0$ and $Y_0$ generate $\sp$. Note that the set of elements $Z$ such that $Z$ and $Y_0$ generate $\sp$ is open in Zariski topology.
		So this set contains an element from the orbit of $X$.  It follows that for the element $X$ there exists an element $Y$ in the orbit of $Y_0$ such that $X$ and $Y$ generate~$\sp$.
	\end{proof}

	\section{Examples}
	Let us give two examples with specific pairs of nilpotent matrices generating $\sp$. 
	\begin{example}
		It follows from Lemma~\ref{lowestnilp} that the matrices $N = \sum\limits_{i=1}^{n-1}{(E_{i,i+1} - E_{n+i+1,n+i})} + E_{n, 2n}$
		and $E_{n+1, 1}$ generate $\sp$.
	\end{example}
	Let $\g$ be a semisimple Lie algebra over an algebraically closed field $\K$ of characteristic zero. Let $\Phi$ be the root system of the algebra~$\g$ and $\Delta$ be a set of simple roots. Denote by $\Phi^{+}$ and $\Phi^{-}$ the set of positive and negative roots of $\Phi$ with respect to $\Delta$. Consider the elements $N = \sum\limits_{\alpha \in \Delta}v_{\alpha}$ and $M = \sum\limits_{\alpha \in \Delta}{\lambda_{\alpha}}v_{-\alpha}$, where $\{\lambda_{\alpha}\}_{\alpha \in\Delta}$ are choosen in such a way that the element $T = [N, M]$ satisfies the following conditions: $\alpha(T) \neq 0$ holds for any $\alpha \in \Delta$, and for all~$\alpha$, $\beta \in \Delta$ the condition $\alpha \neq \beta$ implies $\alpha(T) \neq \beta(T)$. 
	\begin{proposition}\label{proposition}
		The nilpotent elements $N$ and $M$ generate $\g$. 
	\end{proposition}
	\begin{proof}
		Similarly to the proof of Lemma~\ref{Vandermonde}, we obtain that the subspace $\sum\limits_{\alpha \in \Delta}{\g_{\alpha}}$ is contained in the subalgebra $\mathfrak{a}$ generated by the elements $T$ and $N$. It follows that the subalgebra~$\sum\limits_{\alpha \in \Phi^{+}}{\g_{\alpha}}$ is contained in $\mathfrak{a}$.
		Similarly, the subalgebra $\sum\limits_{\alpha \in \Phi^{-}}{\g_{\alpha}}$ is contained in $\mathfrak{a}$ as well. Hence $N$ and $M$ generate~$\g$. 
	\end{proof}
	
	The following example illustrates Proposition~\ref{proposition} for $\g = \sp$. 
	\begin{example}
		Let us consider the matrix $N$ from Example 1 and the matrix
		$$
		M = \begin{pmatrix}
		0    & \ldots & 0 & 0 & 0 & 0  & \ldots & 0 \\
		\a_1   & \dots & 0 & 0 & 0 & 0  & \ldots & 0 \\
		\vdots & \ddots & \vdots & \vdots &\vdots & \vdots & \ddots & \vdots \\
		0  &  \dots & \a_{n-1} & 0 & 0  & 0  & \ldots & 0\\
		0 & 0 &  \dots & 0 & 0 & -\a_1  & \ldots & 0 \\
		\vdots & \vdots & \ddots & \vdots  & \vdots & \vdots & \ddots & \vdots\\
		0 & 0  &  \dots & 0 & 0 & 0  & \ldots&-\a_{n-1} \\
		0 & 0   & \ldots & \a_n & 0 & 0  & \ldots & 0 \\
		\end{pmatrix}, \text{ where $\a_i = 3^i$. }
		$$ 
		
		Define $T = [N, M] = \diag(\a_1,\a_2-\a_1, \dots, \a_n-\a_{n-1}, -\a_1, \a_1 - \a_2, \dots, \a_{n-1} - \a_n)$. Simple roots from $\Delta$ take pairwise different nonzero values on $T$, so $M$ and $N$ satisfy conditions of Proposition~\ref{proposition}.
		
	\end{example}
	
	\section{Conjectures and remarks}
	Let us make some observations on the case of an arbitrary simple Lie algebra $\g$ over an algebraically closed field $\K$ of characteristic zero.
	\begin{conjecture}\label{1}
		For any nonzero nilpotent element $X \in \g$ there exists a nilpotent element $Y \in \g$ such that $X$ and $Y$ generate $\g$.
	\end{conjecture}
	
	The final part of the proof of Theorem~\ref{theorem} holds for any simple Lie algebra $\g$. Hence, in the case of an arbitrary simple Lie algebra the proof of Conjecture~\ref{1} reduces to Lemma~\ref{lowestnilp}. The first idea is to take the same nilpotent $Y$ as in the proof of Lemma~\ref{lowestnilp}, but it does not always work. For example, for $\g=\sl_{2n}(\K)$ the elements $X$ and $Y$ from Lemma~\ref{lowestnilp} generate not~$\g$ but its proper subalgebra $\sp$. It can be shown that the elements $X$ and $Y$ always generate a reductive Lie algebra. In \cite[Section 6]{BK} Kostant showed that the sum of these elements is a semisimple regular element. 
	
	This paper had been prepared when the author found the work \cite{DG}. One can notice that \cite[Theorem~2.8]{DG} implies Conjecture~$1$ for Lie algebras whose system of simple roots has no automorphism. The proof is based on computations in computer algebra system GAP. It is used to verify that the lowest weight vector and the element $Y$ from the proof of Lemma~\ref{lowestnilp} generate the Lie algebra $\g$. 
	
	\smallskip	
	
	Let us now come to a possible application of Theorem~\ref{theorem}. An action of a group $H$ on a set $X$ is called infinitely transitive if for all positive integer $m$ for any two tuples of pairwise distinct points $x_1,\dots,x_m$ and $y_1,\dots,y_m$ in $X$ there exists an element $h \in H$ such that $h\cdot x_i = y_i$ for all $1 \leq i \leq m$.  
	
	Let $\A^n$ be the affine space of dimension $n$ and $\G_a$ be the additive group of the ground field~$\K$. Denote by $\SAut(\A^n)$ the subgroup of $\Aut(\A^n)$ generated by all $\G_a$-subgroups in~$\Aut(\A^n)$. It is easy to see that for $n \geq 2$ the group $\SAut(\A^n)$ acts infinitely transitively on $\A^n$. It is an interesting question whether there exists a finite number of $\G_a$-subgroups in $\SAut(\A^n)$ such that the group generated by them acts infinitely transitively on $\A^n$. Our result in~\cite{AC} enabled to prove the following theorem. 
	\begin{theorem}\cite[Theorem~5.17]{AKZ}
		For all $n \geqslant 2$ one can find three $\G_a$-subgroups $U_1$, $U_2$, $U_3$ in $\SAut(\A^n)$ 
		such that the group $H = \langle U_1, U_2, U_3 \rangle $ acts infinitely transitively on $\A^n$.
	\end{theorem}
	The exponential map establishes a correspondence between nilpotents in a Lie algebra $\g$ and $\G_a$-subgroups in the corresponding algebraic group $G$. Hence Theorem~\ref{theorem} implies that for each $\G_a$-subgroup $U_1$ of  $\operatorname {Sp}_{2n}(\K)$ there exists a $\G_a$-subgroup $U_2$ of  $\operatorname {Sp}_{2n}(\K)$ such that $\langle U_1, U_2 \rangle = \operatorname {Sp}_{2n}(\K)$. In particular, the group $\langle U_1, U_2 \rangle$ acts transitively on $\A^{2n} \setminus \{0\}$.
	
	An algebraic automorphism of $\A^{2n}$ is called a symplectomorphism if its differential preserves a non-degenerate skew-symmetric bilinear form at each point on $\A^{2n}$. Involving one-parameter subgroup of  non-linear symplectomorphisms we can hope for symplectic analogue of \cite[Theorem~5.17]{AKZ}.
	
	\begin{conjecture}
		For all $n \geqslant 2$ one can find three $\G_a$-subgroups $U_1$, $U_2$, $U_3$ of symplectomorphisms of $\A^{2n}$ 
		such that the group $H = \langle U_1, U_2, U_3 \rangle $ acts infinitely transitively on~$\A^{2n}$. 
	\end{conjecture}

\end{document}